\documentclass[leqno,11pt,letterpaper, english]{amsart}
\usepackage[usenames,dvipsnames]{color}
\usepackage[colorlinks=true,linkcolor=Red,citecolor=Green]{hyperref}
\usepackage{amsmath,amssymb,amsthm,graphicx,url}
\usepackage{epstopdf}
\usepackage{color}
\usepackage{dsfont}
\usepackage[T1]{fontenc}
\usepackage{natbib}

\newcommand{\xqedhere}[1]{%
    \rlap{%
         \hbox to#1{%
           \hfil
           \llap{%
               \ensuremath{\square}
           }%
       }%
   }%
}

\def\pasdegrille{\let\grille = \pasgrille}

\def\aat#1#2#3{
\divide \dimen1 by 48 \dimen3=\dimen1 \multiply \dimen1 by #1
\advance \dimen1 by -\dimen3 \divide \dimen1 by 101 \multiply
\dimen1 by 100 \divide \dimen2 by \count11 \multiply \dimen2 by #2
\setbox0=\hbox{#3}\ht0=0pt\dp0=0pt
  \rlap{\kern\dimen1 \vbox to0pt{\kern-\dimen2\box0\vss}}\dimen1= \wd1
\dimen2=\ht1}
\def\pasgrille{
\count12= \dimen1 \divide \count12 by 50 \divide \dimen2 by
\count12 \count11 =\dimen2 \ \divide \dimen1 by 48
\setlength{\unitlength}{\dimen1} \smash{\rlap{\ }} \dimen1= \wd1
\dimen2=\ht1 }
\def\grille{
\count12= \dimen1 \divide \count12 by 50 \divide \dimen2 by
\count12 \count11 =\dimen2 \ \divide \dimen1 by 48
\setlength{\unitlength}{\dimen1}
\smash{\rlap{\graphpaper[1](0,0)(50, \count11)}} \dimen1= \wd1
\dimen2=\ht1 }

\pasdegrille

\newcommand{\be}{\begin{equation}}
\newcommand{\ee}{\end{equation}}

\renewcommand{\Re}{\mathop{\rm Re}\nolimits}
\renewcommand{\Im}{\mathop{\rm Im}\nolimits}

\newcommand{\nn}{\nonumber} 

\theoremstyle{plain}

\newtheorem{thm}{Theorem}

\newtheorem{prop}{Proposition}[section]
\newtheorem{cor}[prop]{Corollary}
\newtheorem{lem}[prop]{Lemma}

\theoremstyle{definition}

\numberwithin{equation}{section}

\def\squarebox#1{\hbox to #1{\hfill\vbox to #1{\vfill}}}

\usepackage{amsxtra}

\ifx\pdfoutput\undefined
  \DeclareGraphicsExtensions{.pstex, .eps}
\else
  \ifx\pdfoutput\relax
    \DeclareGraphicsExtensions{.pstex, .eps}
  \else
    \ifnum\pdfoutput>0
      \DeclareGraphicsExtensions{.pdf}
    \else
      \DeclareGraphicsExtensions{.pstex, .eps}
    \fi
  \fi
\fi

\title[Semi-classical localization of the Schrödinger Resolvent]{Semi-Classical Localization of the Schrödinger Resolvent on Closed Riemann Surfaces}

\author[S.~Campagne]{Sébastien Campagne}
\address{Universit{\'e} Paris-Saclay, Math{\'e}matiques, UMR 8628 du CNRS, B{\^a}t 307, 91405  Orsay Cedex, France,   and Institut Universitaire de France}
\email{Sebastien.campagne@universite-paris-saclay.fr}

\date{\today}

\usepackage{amssymb}
\usepackage{amsmath, amsthm, amsopn, amsfonts}

\usepackage{comment}

\def\11{{\rm 1~\hspace{-1.4ex}l} }
\def\R{\mathbb R}

\begin{document}    

\begin{abstract}
This paper investigates the localization properties of solutions to the semi-classical Schr\"odinger equation on closed Riemann surfaces.
Unlike classical studies that assume a smooth potential, our work addresses the challenges arising from irregular potentials, specifically those that are merely bounded.
We employ a regularization technique to manage the potential's lack of smoothness and establish a local-to-global estimate.
This result provides a quantitative measure of how the local regularity of the potential influences the global concentration of the solution, thereby bridging the gap between smooth and non-continuous regimes in semi-classical analysis.
\begin{center} \rule{0.6\textwidth}{.4pt} \end{center}
\end{abstract}

\ \vskip -1cm \noindent\hfil\rule{0.9\textwidth}{.4pt}\hfil \vskip 1cm 
\maketitle

\section{Introduction}
In this paper, we study the localization properties of solutions \( u \in H^2 \) to the semi-classical Schr\"odinger equation:
\begin{equation}
    (P_V - E)u := (-h^2\Delta + V - E)u = f, 
\end{equation}
where \( V \) is a real-valued, bounded potential defined on a closed Riemann surface \( M \), \( E \in I \subset \mathbb{R} \) is an energy level varying in a compact set \( I \), and \( f \in L^2(M) \) is a source term. 

It is well established that when \( V \) is smooth, one can derive 
the following local-to-global estimate: for any open subset \( U \subset M \) and for sufficiently small \( h \), there exists a constant \( C > 0 \) such that
\begin{equation}
    \int_M |u|^2 + |h\nabla u|^2 \leq C e^{C/h} \left( \int_U |u|^2 + |h\nabla u|^2 + \int_M |f|^2 \right).
\end{equation}
This estimate relies on techniques from microlocal analysis, particularly semiclassical Carleman estimates. Specifically, for \( R > 0 \) and \( E \in I \), there exists \( C > 0 \) such that for all sufficiently small \( h \) and all \( u \in \mathcal{C}_0^\infty(B(0, R)) \),
\begin{equation}
    \int_{B(0, R)} |u|^2 + |h\nabla u|^2 \leq C e^{C/h} \int_{B(0, R)} |(P_V - E)u|^2,
\end{equation}
as proved by \cite{dyatlovMathematicalTheoryScattering2019} (see, e.g., Theorem 2.32). 

In the present work, the potential \( V \) is not assumed to be smooth, precluding 
the direct application of these classical results. We therefore develop alternative methods to address these challenges. Recent studies have shown that Carleman-type estimates can be obtained under weaker regularity assumptions on \( V \). For instance, \cite{kloppSemiclassicalResolventEstimates2019a} established a semiclassical Carleman estimate for potentials \( V \in L^\infty(\mathbb{R}^2, \mathbb{R}) \): for any \( R > 0 \) and \( E \in I \), there exists \( C > 0 \) such that for sufficiently small \( h \) and \( u \in \mathcal{C}_0^\infty(B(0, R)) \),
\begin{equation}
    \int_{B(0, R)} |u|^2 + |h\nabla u|^2 \leq C e^{C/h^{4/3}} \int_{B(0, R)} |(P_V - E)u|^2.
\end{equation}

Subsequently, 
\cite{vodevSemiclassicalResolventEstimates2020} refined this estimate by incorporating the H\"older regularity of \( V \): for an \( \alpha \)-H\"older potential \( V \), it was shown that
\begin{equation}
    \int_{B(0, R)} |u|^2 + |h\nabla u|^2 \leq C e^{C h^{-4/(\alpha+1)}} \int_{B(0, R)} |(P_V - E)u|^2.
\end{equation}

Inspired by Vodev's approach, we establish the following result on closed Riemann surfaces as opposed to the traditional Euclidean plan or non-compact hyperbolic surfaces:
\begin{thm}
    \label{thm: mainthm}
    Let \( M \) be a closed Riemann surface, and let \( U \subset M \) be an open subset. Let \( E \in I \subset \mathbb{R} \) and \( V \in L^\infty(M, \mathbb{R}) \). Then there exist constants \( C > 0 \) and \( h_0 > 0 \) such that for all \( 0 < h < h_0 \) and \( u \in H^2(M) \),
    \begin{equation}
        \int_M |u|^2 + |h\nabla u|^2 \leq C e^{C\beta(h)} \left( \int_U |u|^2 + |h\nabla u|^2 + \int_M |(P_V - E)u|^2 \right),
    \end{equation}
    where
    \begin{equation}
        \beta(h) = \frac{1}{h^{4/3}} \sup_{x_0 \in M} \sup_{x \in B(x_0, h^{2/3} \kappa)} |V(x) - V(x_0)|^{1/2},
    \end{equation}
    and \( \kappa > 0 \) is a fixed small constant.
\end{thm}

The function \( \beta(h) \) quantifies the local regularity of the potential \( V \). When \( V \) is \( \alpha \)-H\"older continuous, we have
\[
\beta(h) \leq C_0 h^{2\alpha/3} \kappa^{\alpha},
\]
and thus,
\begin{equation}
    \int_M |u|^2 + |h\nabla u|^2 \leq C e^{C h^{\frac{\alpha - 4}{3}}} \left( \int_U |u|^2 + |h\nabla u|^2 + \int_M |(P_V - E)u|^2 \right). 
\end{equation}

This recovers the classical smooth case when \( V \) is Lipschitz, although our estimate is slightly weaker than Vodev's. Nevertheless, it also recovers the estimate of Klopp and Vogel for merely bounded and non-continuous potentials \( V \). 

Our proof follows the general strategy of \cite{vodevSemiclassicalResolventEstimates2020}, adapted to the geometry of closed Riemann surfaces and the lower regularity of the potential \( V \).

\section{Proof of theorem \ref{thm: mainthm}}
\subsection{Reduction of the problem}
\label{subsect: Uniformization}
Let \( M \) be a closed Riemannian surface. Thanks to the Poincar\'e uniformization theorem, 
we know that \( M \) is conformally equivalent to a unique closed surface of constant curvature. This surface is a quotient of one of the following covering surfaces by a free action of a discrete subgroup of its isometry group:
\begin{itemize}
    \item the Euclidean plane \( \mathbb{R}^2 \),
    \item the sphere \( S^2 \),
    \item the hyperbolic disk \( \mathbb{H}^2 \). 
\end{itemize}
(see, for instance, the book of \cite{de2016uniformization}).

Consequently, \( M \) can be viewed as a compact quotient of one of these three surfaces. 
These three surfaces can be covered by \( M \) via the free action of the discrete subgroup. We can then view the function \( u \) as a periodic function \( u_{\text{per}} \) on the covering space,
where \( u_{\text{per}} \) satisfies the following periodic Schr\"odinger equation on the covering:
\begin{equation}
    (-h^2\Delta + V_{\text{per}} - E)u_{\text{per}} = f_{\text{per}}.
\end{equation}

This reduces our analysis to just three cases:
\begin{itemize}
    \item \textbf{Case of the Euclidean plane \( \mathbb{R}^2 \):}
      We can replace \( U \) by a small ball \( B(0, R) \) centered at \( 0 \) inside \( M \). 
      It is sufficient to work in a larger ball \( B(0, R_0) \) containing \( M \).

    \item \textbf{Case of the hyperbolic disk \( \mathbb{H}^2 \):}
      We can replace \( U \) by a small ball \( B(0, R) \) centered at \( 0 \) inside \( M \).
      It is sufficient to work in a larger ball \( B(0, R_0) \) inside the disk. The Poincar\'e disk is equipped with the metric
      \begin{equation}
          ds^2 = \frac{4 \sum_i dx_i^2}{(1 - \sum_i x_i^2)^2},
      \end{equation}
      so on \( B(0, R_0) \), the metric is equivalent to a Euclidean metric. Consequently, this case can be treated similarly to the Euclidean plane case.

    \item \textbf{Case of the sphere \( S^2 \):}
      We can replace \( U \) by a small ball \( B(0, R) \) centered at the south pole of \( S^2 \) and consider \( M \) as \( S^2 \). 
      We then remove a smaller ball \( B(0, R/2) \) centered at the south pole inside \( B(0, R) \). 
      The perforated sphere can be unfolded onto a ball \( B(0, R_0) \), so that \( B(0, R) \setminus B(0, R/2) \) becomes a ring \( A(0, \tilde{R}_0, R_0) \) at the edge of \( B(0, R_0) \), where the metric becomes Euclidean.
\end{itemize}

Thus, in each case, we reduce our study to a ball \( B(0, R_0) \). However, we must still consider two subcases: \( U \) is identified with a ball \( B(0, R) \) centered at \( 0 \) inside \( B(0, R_0) \), or \( U \) is identified with an ring \( A(0, \tilde{R}_0, R_0) \) at the edge of \( B(0, R_0) \).

\subsection{Carleman estimate on a ball}
\label{section:CarlemansurBouleetPotentielBornee}
In view of the results from the previous section, let \( R_0 > 0 \), \( R > 0 \), and \( \tilde{R}_0 > 0 \) be such that \( R < R_0 \) and \( \tilde{R}_0 < R_0 \).
Let \( V_{\text{per}} \in L^\infty(B(0, R_0)) \). As a preliminary step, we construct a regularized version of the potential \( V_{\text{per}} \). To do so, we first extend \( V_{\text{per}} \) by \( 0 \) outside \( B(0, R_0) \) to the entire plane \( \mathbb{R}^2 \).
Let \( \varphi \in \mathcal{C}_0^\infty(\mathbb{R}^2, \mathbb{R}) \) be a regularizing kernel with compact support in the ball \( B(0, 1) \). 
We assume that \( \varphi \) satisfies the following properties:
\begin{itemize}
    \item \( \varphi \) has total mass \( 1 \), i.e., \( \int_{\mathbb{R}^2} \varphi(x) \, dx = 1 \),
    \item for \( i \in \{1, 2\} \), \( \partial_{x_i} \varphi \) has total mass \( 0 \), i.e., \( \int_{\mathbb{R}^2} \partial_{x_i} \varphi(x) \, dx = 0 \). 
\end{itemize}
We then define the regularized potential as follows:
\begin{equation}
\label{eq: definition-reg-de-V}
    \Tilde{V}_\theta(x)=\int_{\mathbb{R}^2} \frac{1}{\theta^2} V_{per}(x-t)\varphi(\frac{t}{\theta})dt, \mbox{where } x\in\mathbb{R}^2 \mbox{ and } \theta \in ]0,1]
\end{equation}
$\Tilde{V}_\theta$ is a "smooth version" of $V_{per}$.

\begin{prop}
	\label{prop: def de omega_V}
for $\theta>0$ and $i\in\{1,2\}$:

\begin{equation}
||\Tilde{V}_\theta-V_{per}||_{L^\infty(B(0,R_0))}\leq \omega_V(\theta)
\end{equation}
and:

\begin{equation}
\label{eq: divergence de tilde V theta}
||\partial_{x_i} \Tilde{V}_\theta||_{L^\infty(B(0,R_0))}\leq \frac{1}{\theta}\omega_V(\theta)
\end{equation}
with
\begin{equation}
\omega_V(\theta):=\sup_{x_0\in B(0,R_0)} \sup_{t \in B(0,\theta)}|V_{per}(x_0-t)-V_{per}(x_0)|
\end{equation}
\end{prop}

\begin{proof}

Indeed for $x\in B(0,R_0)$:
\begin{align}
    |\Tilde{V}_\theta(x)-V_{per}(x)| &\leq \int_{\mathbb{R}^2} \frac{1}{\theta^2} |V_{per}(x-t)-V_{per}(x)|\varphi(\frac{t}{\theta})dt\nn\\
    &\leq \sup_{t \in B(0,\theta)}|V_{per}(x-t)-V_{per}(x)| \nn\\
    &\leq \sup_{x_0\in B(0,R_0)} \sup_{t \in B(0,\theta)}|V_{per}(x_0-t)-V_{per}(x_0)| = \omega_V(\theta)
\end{align}

As for the second property, for $i\in\{1,2\}$:

\begin{align}
    |\partial_{x_i}\Tilde{V}_\theta(x)| &\leq \int_{\mathbb{R}^2} \frac{1}{\theta^{3}} |V_{per}(x-t)-V_{per}(x)||\partial_{x_i}\varphi(\frac{t}{\theta})|dt \nn\\
    &\leq \lVert \partial_{x_i}\varphi\rVert_{L^1} \frac{1}{\theta}\sup_{t \in B(0,\theta)}|V_{per}(x-t)-V_{per}(x)| \nn\\
    &\leq \lVert \nabla \varphi \rVert_{L^1} \frac{1}{\theta}\omega_V(\theta)
\end{align}
\end{proof}

If \( V \) is sufficiently regular (i.e., continuous), then \( \tilde{V}_\theta \) converges to \( V_{\text{per}} \) in the \( L^\infty \) norm as \( \theta \to 0 \). However, this convergence cannot hold in general for \( V \in L^\infty \): a smooth function can only converge uniformly to a continuous function.

We therefore need to control the rate at which the derivative of \( \tilde{V}_\theta \) may diverge. Specifically, unless \( V_{\text{per}} \) is constant, the modulus of continuity \( \omega_V(\theta) \) decreases at most as fast as \( \theta \) as \( \theta \to 0 \).

Due to the extension of \( V_{\text{per}} \) by \( 0 \) outside \( B(0, R_0) \), \( \omega_V \) may not tend to \( 0 \). To address this issue, we exploit the periodicity of \( V_{\text{per}} \) and assume that \( V_{\text{per}} \) is originally defined on a larger ball \( B(0, R_1) \) with \( R_1 > R_0 \), and then extended by \( 0 \) outside \( B(0, R_1) \). For sufficiently small \( \theta \), \( \omega_V \) then depends solely on the regularity of \( V_{\text{per}} \) within the ball \( B(0, R_0) \).

Having completed the regularization of \( V_{\text{per}} \), we can now proceed to prove our main results. We begin with the following theorem:

\begin{thm}
\label{thm:EStimCarleman}
Let $(\tilde{V}_\theta)_\theta$ be the family of potentials defined in the equation \ref{eq: definition-reg-de-V}. Then:

\begin{itemize}
\item there is a function $\phi \in C^\infty(\mathbb{R}^2,\mathbb{R})$ positive, $h_0 >0 $, $\theta_0>0$ and $C>0$, such that for all $ h\in \left]0, h_0 \right]$, $\theta \in \left]0, \theta_0 \right] $, $u \in C_0^\infty(B(0,R_0))$:

\begin{align}
    &h\int_{B(0,R_0)} e^{2\phi(x)\frac{\omega_V(\theta)^{1/2}}{h\theta^{1/2}}} \left(|u(x)|^2 +\frac{\theta}{\omega_V(\theta)} |h\nabla_x u(x)|^2\right)dx\nn\\
    &\leq C \frac{\theta^{3/2}}{\omega_V(\theta)^{3/2}}\int_{B(0,R_0)} e^{2\phi(x)\omega_V(\theta)^{1/2}/h\theta^{1/2}}|(-h^2\Delta_x +\Tilde{V}_{\theta}(x)- E)u(x)|^2 dx
\end{align}

\item there is a function $\phi \in C^\infty(\mathbb{R}^2,\mathbb{R})$ positive and radially decreasing, $h_0 >0 $, $\theta_0>0$ and $C>0$, such that for all $ h\in \left]0, h_0 \right]$, $\theta \in \left]0, \theta_0 \right] $, $u \in C_0^\infty(B(0,R_0)\setminus B(0,R/2))$:

\begin{align}
    &h\int_{B(0,R_0)} e^{2\phi(x)\frac{\omega_V(\theta)^{1/2}}{h\theta^{1/2}}} \left(|u(x)|^2 +\frac{\theta}{\omega_V(\theta)} |h\nabla_x u(x)|^2\right)dx\nn\\
    &\leq C \frac{\theta^{3/2}}{\omega_V(\theta)^{3/2}}\int_{B(0,R_0)} e^{2\phi(x)\omega_V(\theta)^{1/2}/h\theta^{1/2}}|(-h^2\Delta_x +\Tilde{V}_{\theta}(x)- E) u(x)|^2 dx
\end{align}
 
\end{itemize}
\end{thm}

In our goal to obtain estimates on Riemann surface, the decreasing case will be used for the real plan and the Hyperbolic disk, the increasing will be used for the sphere.

\begin{proof}
In view of the results obtained previously on $\Tilde{V}_\theta$ (Eq~\ref{eq: divergence de tilde V theta}), let be:

\begin{equation}
    V_\theta := \frac{\theta}{\omega_V(\theta)}\Tilde{V}_\theta
\end{equation}

So its derivative is unlikely to diverge in infinite norms. If $V_{per}$ is not constant on $B(0,R_0)$, the quantity $\theta/\omega_V(\theta)$ is bounded because $\omega_V(\theta)$ can't decrease faster than $\theta$ which reachs $0$. In the case of $V_{per}$ constant , we just take $V_\theta =0$, $\omega_V=Id$ and replace $E$ by $E-V_{per}$. Also, the quantity $\omega_V(\theta)$ will be not $0$ for $\theta$ small enough for any $V_{per}$.

On the same way, we replace $E$ by:

\begin{equation}
    E_\theta := \frac{\theta}{\omega_V(\theta)}E
\end{equation}
Let be $u\in C_0^\infty (B(0,R_0))$ and let be $\phi \in \mathcal{C}^\infty(\R^2,\R)$, we can suppose $V_\theta -E_\theta$ and $\nabla_x V_\theta$ small. Indeed if we take $\delta>0$ small and put $\tilde{h}=h\delta$, then:

\begin{equation}
\label{eq: deltapetit}
e^{\phi/\tilde{h}}(-\tilde{h}^2\Delta_x +\delta^2(V_\theta -E_\theta))e^{-\phi/\tilde{h}}u= \delta^2 e^{\phi/\delta h}(-h^2\Delta_x +V_\theta -E_\theta)e^{-\phi/\delta h}u
\end{equation}
So we just have to replace $\phi$ by $\phi/\delta$ and $C$ by $C\delta^3$ in the theorem \ref{thm:EStimCarleman}.

In the following, we use $\bullet$ to denote the scalar product between vectors of $\mathbb{R}^2$ and $\langle \, , \, \rangle$ to denote the usual scalar product of $L^2$.

Let define the operator $P$:

\begin{align}
    P &:= e^{\phi/h}(-h^2\Delta_x +V_\theta-E_\theta) e^{-\phi/h}\nn\\
    &=-h^2\Delta_x +2h\nabla_x \phi \bullet \nabla_x + h\Delta_x \phi - \nabla_x \phi \bullet \nabla_x \phi +V_{\theta}-E_{\theta}
\end{align}

and let show this small result on $P$:

\begin{lem}
\label{lemme: Ineg de separation avec V}
Let $u\in C_0^\infty (B(0,R_0))$, $\forall h>0$:

\begin{align}
    \frac{1}{h} \lVert Pu\rVert_{L^2}^2 &\geq \frac{1}{h}\langle\left[ (e^{\phi/h}P_0e^{-\phi/h})^\ast  , e^{\phi/h}P_0e^{-\phi/h}  \right]u,u\rangle\nn\\
    &-4\lVert \nabla_x V_\theta \bullet \nabla_x\phi \rVert_\infty \lVert u\rVert_{L^2}^2    
\end{align}
Where $P_0=-h^2\Delta_x$.
\end{lem}

\begin{proof}
Indeed:
\begin{equation}
	 \frac{1}{h} \lVert Pu\rVert^2 \geq \frac{1}{h} (\lVert Pu\rVert^2 -\lVert P^\ast  u\rVert^2)
\end{equation}
With $P^\ast $ the adjoint of $P$: 

\begin{align}
	P^\ast  &= -h^2\Delta_x -(2h \Delta_x \phi +2h\nabla_x \phi \bullet \nabla_x) +h\Delta_x \phi - \nabla_x \phi \bullet \nabla_x \phi +V_{\theta}-E_{\theta} \nn\\
\end{align}

So:
\begin{align}
    \frac{1}{h} \lVert Pu\rVert^2 &\geq \frac{1}{h} (\lVert Pu\rVert^2 -\lVert P^\ast  u\rVert^2) \nn\\
    &=\frac{1}{h} (\lVert (e^{\phi/h} P_0e^{-\phi/h}) u\rVert^2  -\lVert (e^{\phi/h} P_0e^{-\phi/h})^* u\rVert^2) \nn\\
    &\quad + \frac{2}{h}(  Re\langle(V_{\theta}-E_{\theta})u,(2h\nabla_x \phi \bullet \nabla_x)u\rangle  + Re\langle(V_{\theta}-E_{\theta})u, (2h \Delta_x \phi +2h\nabla_x \phi \bullet \nabla_x)u\rangle) \nn\\
    &=\frac{1}{h}(\langle (e^{\phi/h} P_0e^{-\phi/h})^* (e^{\phi/h} P_0e^{-\phi/h}) ,u\rangle -\langle (e^{\phi/h} P_0e^{-\phi/h}) (e^{\phi/h} P_0e^{-\phi/h})^* ,u\rangle ) \nn\\
    &\quad\; + 4Re(  \langle(V_{\theta}-E_{\theta})u,(\nabla_x \phi \bullet \nabla_x)u\rangle  + \langle(V_{\theta}-E_{\theta})u, ( \Delta_x \phi +\nabla_x \phi \bullet \nabla_x)u\rangle) \nn\\
\end{align}
Because every terms are real, we don't need to take the real part. By integration by part in $\langle(V_{\theta}-E_{\theta})u,(\nabla_x \phi \bullet \nabla_x)u\rangle$, we obtain finally:
 
\begin{align}
    \frac{1}{h} \lVert Pu\rVert_{L^2}^2 &\geq \frac{1}{h}\langle\left[ (e^{\phi/h} P_0e^{-\phi/h})^\ast  ,(e^{\phi/h} P_0e^{-\phi/h})  \right]u,u\rangle\nn\\
    &-4\langle(\nabla_x V_\theta\bullet \nabla_x \phi) u,u\rangle
\end{align}

and it is easy to see that:
$$
\langle(\nabla_x V_\theta\bullet \nabla_x \phi) u,u\rangle \leq 
\lVert \nabla_x V_\theta \bullet \nabla_x\phi \rVert_\infty \lVert u\rVert_{L^2}^2
$$
\end{proof}

So with the lemma \ref{lemme: Ineg de separation avec V}, we can focus on the commutator term:
$$\left[ (e^{\phi/h}P_0e^{-\phi/h})^\ast  , e^{\phi/h}P_0e^{-\phi/h}  \right].$$
As the same way to prove Carleman estimate, we would like to find $d>0$, $h_0>0$ and $C_0>0$ such that

\begin{equation}
 \frac{1}{h} \langle\left[ (e^{\phi/h}P_0e^{-\phi/h})^\ast  , e^{\phi/h}P_0e^{-\phi/h}  \right]u,u\rangle + d\|(e^{\phi/h}P_0e^{-\phi/h})u \|_{L^2}^2 \geq C_0 \|u\|_{H^1}^2
 \label{eq-carlemanlaplacienlibre}
\end{equation}

for $0<h<h_0$.  
So in the same way as Carleman estimate,
we construct $\phi$ such that it verify the Hörmander condition:

\begin{equation}
p_\phi(x,\xi)=0 \Longrightarrow\{Re\, p_\phi, Im\, p_\phi \}(x, \xi)>0
\end{equation}

where $p_\phi$ is the principal symbol of $P_{\phi}:=e^{\phi/h}P_0e^{-\phi/h}$:

\begin{equation}
	p_\phi(x,\xi)=(\xi +i\nabla_x \phi(x))^2.
\end{equation}

We would like to construct $\phi$ as a radial function at $0$. This would be easy when $u$ has support in $B(0,R_0)\setminus B(0,R/2)$. But if the support can be in the whole $B(0,R_0)$, $\phi$ would have a critical point in $0$ (Fig.~\ref{fig:PointCritiqueFonctionRadiale}). This point causes the cancellation of $\{Im\, p_\phi, Re\, p_\phi \}$ when $p_\phi(0,\xi)=0$. So there will be a difficulty here for the increase case. Let's start by this one.
 
\begin{figure}
    \centering
    \includegraphics[width=0.6\textwidth]{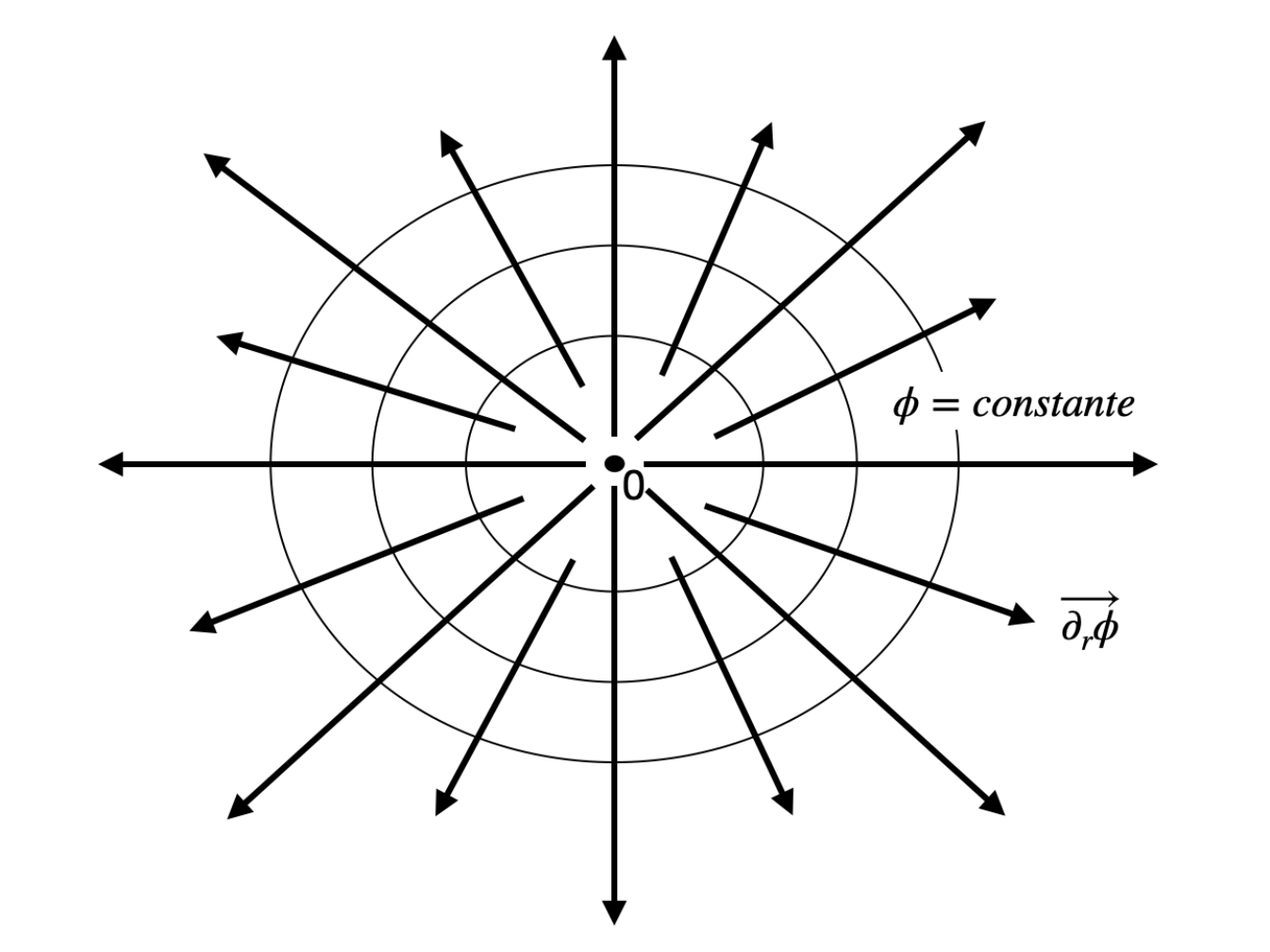}
    \caption{Critical point of a radial function: $\phi$ is a radial function at $0$, so $\phi$ is constant on circles centred at 0 and evolves orthogonally to these circles. Consequently, in $0$, the derivative vector should be $0$.}
    \label{fig:PointCritiqueFonctionRadiale}
\end{figure}

To solve this construction problem on $\phi$, we take two points $x_1$ and $x_2$ of $B(0,R_0)$ and pose:

\begin{equation}
\label{eq: definition-de-phi}
   e^{2\phi(x)/h}:= e^{2\phi_{x_1}(x)/h}\chi_{1}^2(x)+e^{2\phi_{x_2}(x)/h}\chi_{2}^2(x)
\end{equation}

where $\phi_{x_1}(x)=c \, e^{\|x-x_1\|}$, $\phi_{x_2}(x)=c \, e^{\|x-x_2\|}$ with $c>0$. $ \chi_1$ is $0$ on a ball $B_1$ center on $x_1$ and $1$ outside a ball $\tilde{B}_1$ center on $x_1$, likewise $\chi_2$ is $0$ on a ball $B_2$ center on $x_2$ and $1$ outside a ball $\tilde{B}_2$ center on $x_2$. We take $\tilde{B}_1$ and $\tilde{B}_2$ such that $\tilde{B}_1\cap \tilde{B}_2 = \emptyset$ (Fig.~\ref{fig:ConstructionDePhi}).

\begin{figure}
    \centering
    \includegraphics[width=0.6\textwidth]{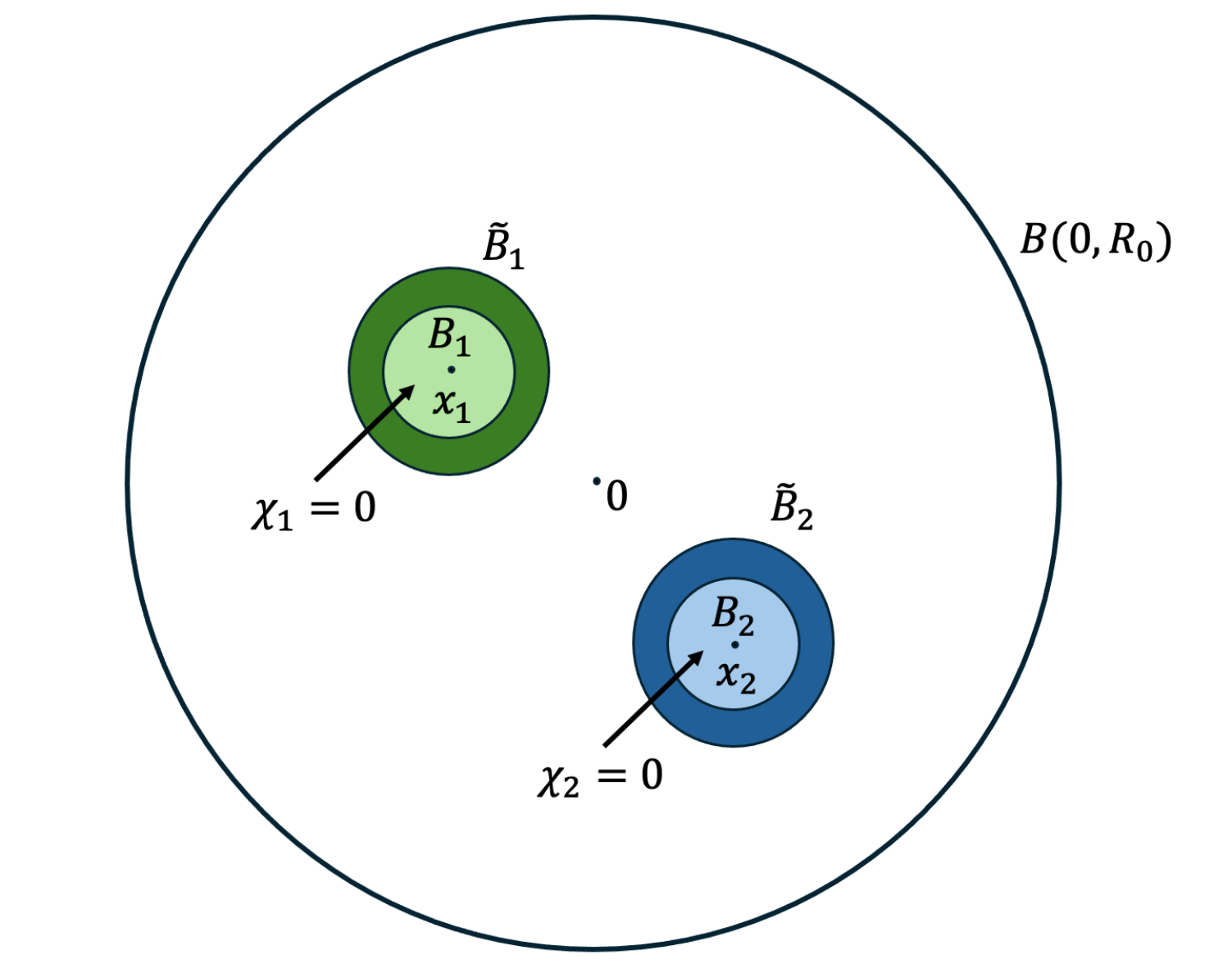}
    \caption{Partition of the ball $B(0,R_0)$ to construct $\phi$.
    }
    \label{fig:ConstructionDePhi}
\end{figure}

So, instead of considering a Carleman inequality on the ball $B(0,R_0)$, this interpolation leads us to consider two inequalities on deformed rings.

 To simplify the next calculations, we assume that we are now working on a ring $A(0,r_1,r_2)$ centered in $0$. Let $\phi$ be $\phi(r)= c\,e^{r}$. In polar coordinate, the Laplacian is:

\begin{equation}
	\label{eq: laplacien polaire}
    P_0= -h^2\partial_r^2 -\frac{1}{r}h^2\partial_r -\frac{h^2}{r^2}\Delta_{S^{1}}
\end{equation}

where $\Delta_{S^{1}}$ is the Beltrami's laplacian on the circle $S^1$.
Its principal symbol is:

\begin{equation}
    p_0(r,\rho)= \rho^2 + \frac{1}{r^2}\omega^2, \, (r,\rho)\in [r_1,r_2]\times\mathbb{R_+}
\end{equation}

where $\omega^2 >0$ is the principal symbol of $-h^2\Delta_{S^{1}}$. By conjugating by $e^{\phi/h}$, we obtain the following symbol:

\begin{equation}
    p_{\phi}(r,\rho)= (\rho+ i\partial_{r}\phi(r))^2 + \frac{1}{r^2}\omega^2
\end{equation}

that is:

\begin{align}
    \left\{
    \begin{array}{ll}
        \Re(p_{\phi})(r,\rho)= & \rho^2 - (\partial_{r}\phi(r))^2+ \frac{1}{r^2}\omega^2 \nn \\
        \Im(p_{\phi})(r,\rho)= & 2\rho \, \partial_r\phi(r) \nn
    \end{array}
\right.
\end{align}

And so the bracket takes the form:
\begin{equation}
    \frac{1}{4}\{\Re(p_\phi),\Im(p_\phi)\}(r,\rho)= \rho^2\partial_r^2\phi(r) +\left(\partial_r \phi(r)\partial_r^2\phi(r)+\frac{\omega^2}{r^3}\right)\partial_r\phi(r)
\end{equation} 

that is: 

\begin{equation}
    \frac{1}{4}\{\Re(p_\phi),\Im(p_\phi)\}(r,\rho)= c\,\rho^2e^{r} +c\,e^{r}\left(c^2e^{2r}+\frac{\omega^2}{r^3}\right) > c\,\rho^2 +c^3
\end{equation}
and so when $p_\phi=0$ (that is when $\rho=0$ and $\omega^2=r^2 \partial_r \phi$), then:

\begin{equation}
    \frac{1}{4}\{\Re(p_\phi),\Im(p_\phi)\}(r,\rho)\geq c^3
\end{equation}

At the same time, we bound by below $p_{\phi}$  when $\rho>>1$:

\begin{equation}
    |p_{\phi}(r,\rho)|^2\geq |\Re(p_\phi)(r,\rho)|^2\simeq \rho^4
\end{equation}

And so for $d>0$ and $c>0$ large enough:

\begin{equation}
    2\{\Re(p_\phi),\Im(p_\phi)\}(r,\rho) +d |p_{\phi}(r,\rho)|^2 \geq C(\rho^2+1)^2
\end{equation}

where $C>0$. It should be noted that \\
$2h\{\Re(p_\phi),\Im(p_\phi)\}(r,\rho)$ is the principal symbol of $\left[ P_{\phi}^\ast  , P_{\phi} \right]$ where $P_\phi=e^{\phi/h}P_0e^{-\phi/h}$.

With symbolic algebra and Gårding inequality (see for instance \cite{alinhac2007pseudo}), we obtain:

\begin{equation}
\label{eq:Carleman_laplacien_libre}
    \frac{1}{h}\langle\left[ P_{\phi}^\ast  , P_{\phi}  \right]u,u\rangle +d\|P_{\phi} u \|^2 \geq C_0\|u\|_{H^1}^2
\end{equation}

where $C_0>0$, $u \in \mathcal{C}_0^\infty(A(0,r_1,r_2))$ and $h$ small.

Now we need to add the residuals terms  with $V_{\theta}-E_{\theta}$ and $\nabla_x V_\theta$. So we take $u \in \mathcal{C}_0^\infty(A(0,r_1,r_2))$. By the lemma \ref{lemme: Ineg de separation avec V}, we show:

\begin{align}
    \frac{1}{h} \lVert Pu\rVert^2 +2d\lVert Pu\rVert^2 &\geq \frac{1}{h}\langle\left[ P_{\phi}^\ast  , P_{\phi}  \right]u,u\rangle + 2d(\lVert P_{\phi} u\rVert-\lVert (V_{\theta}-E_{\theta})u\rVert)^2 \nn\\
    &-4\lVert \nabla_x V_\theta\bullet \nabla_x\phi \rVert_\infty \lVert u\rVert^2  \nn\\
    &\geq\frac{1}{h}\langle\left[ (P_{\phi}^\ast  , P_{\phi}  \right]u,u\rangle + 2d\lVert P_{\phi} u\rVert^2+2d\lVert (V_{\theta}-E_{\theta})u\rVert^2  \nn\\
    &-4d\lVert P_{\phi} u\rVert\lVert (V_{\theta}-E_{\theta})u\rVert-4\lVert \nabla_x V_\theta\bullet \nabla_x\phi \rVert_\infty \lVert u\rVert^2  \nn\\
    &\geq\frac{1}{h}\langle\left[ P_{\phi}^\ast  , P_{\phi}  \right]u,u\rangle + d\lVert P_{\phi} u\rVert^2-2d\lVert (V_{\theta}-E_{\theta})u\rVert^2  \nn\\
    &-4\lVert \nabla_x V_\theta\bullet \nabla_x\phi \rVert_\infty \lVert u\rVert^2  \nn\\
    &\geq C_0\|u\|_{H^1}^2 -4\lVert \nabla_x V_\theta\bullet \nabla_x\phi \rVert_\infty \lVert u\rVert^2 - 2d\lVert (V_{\theta}-E_{\theta})\rVert_\infty^2 \lVert u\rVert^2
\end{align}

We've used the following inequality: $2|a||b| \leq \frac{1}{2}|a|^2 + 2|b|^2$.
We observe that the residual terms are in the form $ \mathcal{O}_{\|V_{\theta}\|_{\infty},\|\nabla_x V_{\theta}\|_{\infty},E_\theta}(1)\lVert u\rVert^2$ . By the assumption that $\|V_{\theta}-E_\theta\|_{\infty}$ and $\|\nabla_x V_{\theta}\|_{\infty}$ are small enough (see eq:~\ref{eq: deltapetit}), we can absorb the residual terms in the $H^1$ norm of $u$. We therefore deduce that on a ring $A(0,r_1,r_2)$, we have:

\begin{equation}
 \tilde{C}_0\lVert Pu\rVert^2 \geq h \|u\|_{H^1}^2
\end{equation}
for $0<h<h_0$ and $\tilde{C}_0>0$. This result adapts to the case of our deformed rings $B(0,R_0)\setminus B_1$ and $B(0,R_0)\setminus B_2$.

Finally, we need to glue together the inequalities obtained on the two rings. To do this, replace $u$ by $e^{\phi/h}u$ with $u\in \mathcal{C}_0^\infty(B(0,R_0))$ and by the definition of $\phi$ (\ref{eq: definition-de-phi}):

\begin{align}
    h\int_{B(0,R_0)} e^{2\phi(x)/h} &(|u(x)|^2 +|h\nabla_x u(x)|^2)dx \nn\\
    &= h\int_{B(0,R_0)} e^{2\phi_{x_1}(x)/h} \chi_1^2(x) (|u(x)|^2 +|h\nabla_x u(x)|^2)dx \nn\\
    &\quad +h\int_{B(0,R_0)} e^{2\phi_{x_2}(x)/h} \chi_2^2(x) (|u(x)|^2 +|h\nabla_x u(x)|^2)dx\nn\\
    &\leq \tilde{C}_0\int_{B(0,R_0)} e^{2\phi_{x_1}(x)/h} \chi_1^2(x) |(-h^2\Delta_x +V_{\theta}(x))u(x)|^2 dx   \nn\\
    &\quad + \tilde{C}_0\int_{B(0,R_0)} e^{2\phi_{x_2}(x)/h} \chi_2^2(x) |(-h^2\Delta_x +V_{\theta}(x))u(x)|^2 dx   \nn\\
    &\quad +  \int_{B(0,R_0)} e^{2\phi_{x_1}(x)/h}  (\tilde{C}_0|[-h^2\Delta_x, \chi_1]u(x)|^2 + |[-h\nabla_x, \chi_1]u(x)|^2)dx   \nn\\
    &\quad+  \int_{B(0,R_0)} e^{2\phi_{x_2}(x)/h}  (\tilde{C}_0|[-h^2\Delta_x, \chi_2]u(x)|^2 + |[-h\nabla_x, \chi_2]u(x)|^2) dx   \nn\\
\end{align}

In the last two right-hand integrals, the terms $|[-h^2\Delta_x, \chi_i]u(x)|^2$ and $|[-h\nabla_x, \chi_i]u(x)|^2$ are majorized by $C_i h^2 (|u|^2 +|h\nabla_x u|^2)$ with $C_i>0$ and the integrations are done on the sets $Supp(\nabla_x \chi_i)\subset \tilde{B}_i\setminus B_i$.\\
 By construction $\phi_1$ is smaller than $\phi_2$ on $ \tilde{B}_1\setminus B_1$ and vice versa. And since $ \tilde{B}_1\setminus B_1 \subset Supp(\chi_2)$ (respectively $ \tilde{B}_2\setminus B_2 \subset Supp(\chi_1)$) the last two right-hand integrals can be absorbed by the first two right-hand integrals for $h$ sufficiently small:

\begin{align}
\int_{B(0,R_0)} e^{2\phi_{x_1}(x)/h}  &(\tilde{C}_0|[-h^2\Delta_x, \chi_1]u(x)|^2 + |[-h\nabla_x, \chi_1]u(x)|^2)dx \nn\\
&=\mathcal{O}_h \left(h\int_{B(0,R_0)} e^{2\phi_{x_2}(x)/h} \chi_2^2(x) (|u(x)|^2 +|h\nabla_x u(x)|^2)dx \right)
\end{align}

\begin{align}
\int_{B(0,R_0)} e^{2\phi_{x_2}(x)/h}  &(\tilde{C}_0|[-h^2\Delta_x, \chi_2]u(x)|^2 + |[-h\nabla_x, \chi_2]u(x)|^2)dx \nn\\
&=\mathcal{O}_h \left(h\int_{B(0,R_0)} e^{2\phi_{x_1}(x)/h} \chi_1^2(x) (|u(x)|^2 +|h\nabla_x u(x)|^2)dx \right)
\end{align}

Finally, there is $h_0>0$ and $C>0$ such that for all $0<h<h_0$ and $u\in \mathcal{C}_0^\infty(B(0,R_0))$:

\begin{equation}
    h\int_{B(0,R_0)} e^{2\phi(x)/h} (|u(x)|^2 +|h\nabla_x u(x)|^2)dx \leq C\int_{B(0,R_0)} e^{2\phi(x)/h}|(-h^2\Delta_x +V_{\theta}(x)- E_{\theta})u(x)|^2 dx
\end{equation}

If we replace $h^2$ by $\theta\, h^2 / \omega_V(\theta)$, we obtain:

\begin{align}
    &h\int_{B(0,R_0)} e^{2\phi(x)\omega_V(\theta)^{1/2}/h\theta^{1/2}} (|u(x)|^2 +\theta |h\nabla_x u(x)|^2/\omega_V(\theta))dx\nn\\
    &\leq C \frac{\theta^{3/2}}{\omega_V(\theta)^{3/2}}\int_{B(0,R_0)} e^{2\phi(x)\omega_V(\theta)^{1/2}/h\theta^{1/2}}|(-h^2\Delta_x +\Tilde{V}_{\theta}(x)- E)u(x)|^2 dx
\end{align}

for $u \in \mathcal{C}_0^\infty(B(0,R_0))$.
which is the first part of the theorem \ref{thm:EStimCarleman}.

For the second case, $u$ has support in $B(0,R_0)\setminus B(0,R/2)$ so we will be far from $0$ and so we can considerate a radial function for $\phi$ and start at the equation \ref{eq: laplacien polaire}. Let's take $\phi=r^{-1}$. So when $p_\phi=0$ (that is when $\rho=0$ and $\omega^2=r^2 \partial_r \phi$), we have:

\begin{equation}
	\{Re (p_\phi), Im (p_\phi) \}= r^{-7} \geq R_0^{-7}>0
\end{equation}

The proof then proceeds in the same way as before on only one ring. So we obtain:

\begin{align}
    &h\int_{B(0,R_0)} e^{2\phi(x)\omega_V(\theta)^{1/2}/h\theta^{1/2}} (|u(x)|^2 +\theta |h\nabla_x u(x)|^2/\omega_V(\theta))dx\nn\\
    &\leq C \frac{\theta^{3/2}}{\omega_V(\theta)^{3/2}}\int_{B(0,R_0)} e^{2\phi(x)\omega_V(\theta)^{1/2}/h\theta^{1/2}}|(-h^2\Delta_x +\Tilde{V}_{\theta}(x)- E)u(x)|^2 dx
\end{align}

for $h$ small enough and $u \in \mathcal{C}_0^\infty(B(0,R_0)\setminus B(0,R/2))$.

\end{proof}

Building on Theorem~\ref{thm:EStimCarleman}, we now establish the following corollary for \( V_{per} \):

\begin{cor}
\label{cor: cor1}
\hfill\\
\begin{itemize}
\item There is a function $\tilde{\phi} \in C^\infty(\mathbb{R}^2,\mathbb{R})$ positive, $h_0 >0 $ and $C>0$, such that for all $ h\in \left]0, h_0 \right]$, $u \in C_0^\infty(B(0,R_0))$:
\begin{align}
    &\int_{B(0,R_0)} e^{2\tilde{\phi}(x)\,\tilde{\omega}_V(h)^{1/2}/ h^{4/3}} (|u(x)|^2 +|h\nabla_x u(x)|^2)dx\nn\\
    &\leq C \frac{1}{\tilde{\omega}_V(h)^{1/2}h^{2/3}}\int_{B(0,R_0)} e^{2\tilde{\phi}(x)\,\tilde{\omega}_V(h)^{1/2}/ h^{4/3}}|(-h^2\Delta_x + V_{per}- E)u|^2dx
\end{align}
with $\tilde{\omega}_V(h)=\omega_V( h^{2/3}\kappa)$, $\kappa>0$.
 
\item There is a function $\tilde{\phi} \in C^\infty(\mathbb{R}^2,\mathbb{R})$ positive and radially decreasing, $h_0 >0 $ and $C>0$, such that for all $ h\in \left]0, h_0 \right]$, $u \in C_0^\infty(B(0,R_0)\setminus B(0,R/2))$:
\begin{align}
    &\int_{B(0,R_0)} e^{2\tilde{\phi}(x)\,\tilde{\omega}_V(h)^{1/2}/ h^{4/3}} (|u(x)|^2 +|h\nabla_x u(x)|^2)dx\nn\\
    &\leq C \frac{1}{\tilde{\omega}_V(h)^{1/2}h^{2/3}}\int_{B(0,R_0)} e^{2\tilde{\phi}(x)\,\tilde{\omega}_V(h)^{1/2}/ h^{4/3}}|(-h^2\Delta_x + V_{per}- E)u|^2dx
\end{align}
with $\tilde{\omega}_V(h)=\omega_V( h^{2/3}\kappa)$, $\kappa>0$.
\end{itemize}
\end{cor}

\begin{proof}

Indeed from the theorem \ref{thm:EStimCarleman} we have:

\begin{align}
    &h\int_{B(0,R_0)} e^{2\phi(x)\omega_V(\theta)^{1/2}/h\theta^{1/2}} (|u(x)|^2 +\theta |h\nabla_x u(x)|^2/\omega_V(\theta))dx\nn\\
    &\leq C \frac{\theta^{3/2}}{\omega_V(\theta)^{3/2}}\int_{B(0,R_0)} e^{2\phi(x)\omega_V(\theta)^{1/2}/h\theta^{1/2}}|(-h^2\Delta_x +\Tilde{V}_{\theta}(x)- E^2)u(x)|^2 dx
\end{align}

for  $u \in C_0^\infty(B(0,R_0))$ (or  $u \in C_0^\infty(B(0,R_0)\setminus B(0,R/2))$), $ h\in \left]0, h_0 \right]$ and $\theta \in \left]0, \theta_0 \right] $. Then:

\begin{align}
     &h\int_{B(0,R_0)} e^{2\phi(x)\omega_V(\theta)^{1/2}/h\theta^{1/2}} (|u(x)|^2 +\theta |h\nabla_x u(x)|^2/\omega_V(\theta))dx\nn\\
     &\leq 2C \frac{\theta^{3/2}}{\omega_V(\theta)^{3/2}}\int_{B(0,R_0)} e^{2\phi(x)\,\omega_V(\theta)^{1/2}/h\theta^{1/2}}|(-h^2\Delta_x + V_{per}(x)- E)u(x)|^2 dx\nn\\
     & \quad +2C \frac{\theta^{3/2}}{\omega_V(\theta)^{3/2}} \int_{B(0,R_0)} e^{2\phi(x)\,\omega_V(\theta)^{1/2}/h\theta^{1/2}}|(\Tilde{V}_\theta(x) -V_{per}(x))u(x)|^2dx\nn\\
     &\leq  2C \frac{\theta^{3/2}}{\omega_V(\theta)^{3/2}}\int_{B(0,R_0)} e^{2\phi(x)\,\omega_V(\theta)^{1/2}/h\theta^{1/2}}|(-h^2\Delta_x + V_{per}(x)- E)u(x)|^2dx\nn\\
     & \quad + 2C \theta^{3/2}\omega_V(\theta)^{1/2} \int_{B(0,R_0)} e^{2\phi(x)\,\omega_{V}(\theta)^{1/2}/h\theta^{1/2}}|u(x)|^2 dx
     \label{eq: def_theta}
\end{align}

by definition of $\omega_{V}$ \ref{prop: def de omega_V}. We take $\theta = h^{2/3}\kappa$ with $\kappa>0$ small enough such that the last terme can be absorb by the first one. We obtain then:

\begin{align}
    &\int_{B(0,R_0)} e^{2\tilde{\phi}(x)\,\tilde{\omega}_V(h)^{1/2}/ h^{4/3}} (|u(x)|^2 +|h\nabla_x u(x)|^2)dx\nn\\
    &\leq  \frac{\tilde{C}}{\tilde{\omega}_V(h)^{1/2}h^{2/3}}\int_{B(0,R_0)} e^{2\tilde{\phi}(x)\,\tilde{\omega}_V(h)^{1/2}/ h^{4/3}}|(-h^2\Delta_x + V_{per}- E)u|^2dx
\end{align}

with $\tilde{C}>0$, $ h\in \left]0, h_0 \right]$, $\theta \in \left]0, \theta_0 \right] $, $u\in \mathcal{C}_0^\infty (B(0,R_0))$ (or $u\in \mathcal{C}_0^\infty (B(0,R_0)\setminus B(0,R/2))$ in the second case) and where $\tilde{\omega}_V(h)=\omega_V( h^{2/3}\kappa)$, $\tilde{\phi}(x)=\phi(x)/\kappa^{1/2}$. The choice of $\kappa$ is necessary only when $V_{per}$ is not continuous on $B(0,R_0)$ (that is $\omega_V$ don't decrease to $0$ when $h$ tend to $0$).

\end{proof}

In Equation~\eqref{eq: def_theta}, we initially choose $\theta = h^{2/3}\kappa$ with $\kappa > 0$. However, if $V$ is $\alpha$-H\"older continuous, this choice can be refined. Specifically, we require that
$$2C \theta^{3/2} \omega_V(\theta)^{1/2} \ll h.$$

Given that $V$ is $\alpha$-H\"older continuous, we have $\omega_V(\theta)^{1/2} \leq \theta^{\alpha/2}$. Consequently, it suffices to ensure that
$$\theta \ll \left(\frac{h}{2C}\right)^{2/(\alpha+3)}.$$

We thus set $\theta = h^{2/(\alpha+3)}\kappa$, where $\kappa > 0$ is sufficiently small. As a result, we obtain the following bound:
$$\frac{\omega_V(\theta)^{1/2}}{h\theta^{1/2}} \leq \tilde{\kappa} h^{-4/(\alpha+3)},$$
for some $\tilde{\kappa} > 0$. This recovers the estimate established in~\cite{vodevSemiclassicalResolventEstimates2020}.

To extend these results to the periodic function $u_{\text{per}}$, we must eliminate the assumption of compact support. This naturally leads us to the following corollary:

\begin{cor}
\label{cor: cor2}
\hfill\\
\begin{itemize}
\item There is a function $\tilde{\phi} \in C^\infty(\mathbb{R}^2,\mathbb{R})$ positive, $h_0 >0 $ and $C>0$, such that for all $ h\in \left]0, h_0 \right]$, $u \in H^2(B(0,R_0))$:
\begin{align}
    &\int_{B(0,\tilde{R}_0)} e^{2\tilde{\phi}(x)\,\tilde{\omega}_V(h)^{1/2}/ h^{4/3}} (|u(x)|^2 +|h\nabla_x u(x)|^2)dx\nn\\
    &\leq C \frac{1}{\tilde{\omega}_V(h)^{1/2}h^{2/3}}\int_{B(0,R_0)} e^{2\tilde{\phi}(x)\,\tilde{\omega}_V(h)^{1/2}/ h^{4/3}}|(-h^2\Delta_x + V_{per}- E)u|^2dx\nn\\
    &\quad +C \frac{h^{4/3}}{\tilde{\omega}_V(h)^{1/2}}\int_{A(0,\tilde{R}_0,R_0)} e^{2\tilde{\phi}(x)\,\tilde{\omega}_V(h)^{1/2}/ h^{4/3}} (|u(x)|^2 +|h\nabla_x u(x)|^2)dx
\end{align}
with $\tilde{\omega}_V(h)=\omega_V( h^{2/3}\kappa)$, $\kappa>0$ and $A(0,\tilde{R}_0,R_0)$ is a ring centre on $0$ with radius $\tilde{R}_0<R_0$.
 
\item There is a function $\tilde{\phi} \in C^\infty(\mathbb{R}^2,\mathbb{R})$ positive and radially decreasing, $h_0 >0 $ and $C>0$, such that for all $ h\in \left]0, h_0 \right]$, $u \in H^2(B(0,R_0)\setminus B(0,R/2))$:
\begin{align}
    &\int_{A(0,R_1,R_2)} e^{2\tilde{\phi}(x)\,\tilde{\omega}_V(h)^{1/2}/ h^{4/3}} (|u(x)|^2 +|h\nabla_x u(x)|^2)dx\nn\\
    &\leq C \frac{1}{\tilde{\omega}_V(h)^{1/2}h^{2/3}}\int_{B(0,R_0)} e^{2K\,\tilde{\omega}_V(h)^{1/2}/ h^{4/3}}|(-h^2\Delta_x + V_{per}- E)u|^2dx\nn\\
    &\quad + C \frac{h^{4/3}}{\tilde{\omega}_V(h)^{1/2}}\int_{B(0,R)} e^{2K\,\tilde{\omega}_V(h)^{1/2}/ h^{4/3}} (|u(x)|^2 +|h\nabla_x u(x)|^2)dx\nn\\
    &\quad +C \frac{h^{4/3}}{\tilde{\omega}_V(h)^{1/2}}\int_{A(0,R_2,R_0)} e^{2\tilde{\phi}(x)\,\tilde{\omega}_V(h)^{1/2}/ h^{4/3}} (|u(x)|^2 +|h\nabla_x u(x)|^2)dx
\end{align}
with $\tilde{\omega}_V(h)=\omega_V( h^{2/3}\kappa)$, $\kappa>0$, $K>\tilde{\phi}$ on $A(0,R_1,R_0)$, $B_r$ a small ball centre on $0$ with radius $r$ and $A(0,R_1,R_2)$ is a ring centre on $0$ with radius $R/2<R_1<R<R_2<R_0$.
\end{itemize}
\end{cor}

\begin{proof}
	By density, of smooth functions in $H^2(B(0,R_0))$, it is sufficient to show the result for smooth functions. 
Let's start with the first point. Thanks to the corollary \ref{cor: cor1} we know that for $u \in C_0^\infty(B(0,R_0))$, we have for $h$ small enough:

\begin{align}
    &\int_{B(0,R_0)} e^{2\tilde{\phi}(x)\,\tilde{\omega}_V(h)^{1/2}/ h^{4/3}} (|u(x)|^2 +|h\nabla_x u(x)|^2)dx\nn\\
    &\leq C \frac{1}{\tilde{\omega}_V(h)^{1/2}h^{2/3}}\int_{B(0,R_0)} e^{2\tilde{\phi}(x)\,\tilde{\omega}_V(h)^{1/2}/ h^{4/3}}|(-h^2\Delta_x + V_{per}- E)u|^2dx
\end{align}

Let's take $u\in \mathcal{C}^\infty(B(0,R_0))$  and let $\chi$ be in $ C_0^\infty(B(0,R_0))$. Then we can apply corollary \ref{cor: cor1} to $\chi u$:

\begin{align}
    &\int_{B(0,R_0)} e^{2\tilde{\phi}(x)\,\tilde{\omega}_V(h)^{1/2}/ h^{4/3}} (|\chi u(x)|^2 +|h\nabla_x \chi u(x)|^2)dx\nn\\
    &\leq C \frac{1}{\tilde{\omega}_V(h)^{1/2}h^{2/3}}\int_{B(0,R_0)} e^{2\tilde{\phi}(x)\,\tilde{\omega}_V(h)^{1/2}/ h^{4/3}}|(-h^2\Delta_x + V_{per}- E)\chi u|^2dx
\end{align}

So if we take $\chi=1$ on $B(0,\tilde{R}_0)$ with $\tilde{R}_0<R_0$ and $\chi\geq 0$, then:

\begin{align}
    &\int_{B(0,\tilde{R}_0)} e^{2\tilde{\phi}(x)\,\tilde{\omega}_V(h)^{1/2}/ h^{4/3}} ( |u(x)|^2 +|h\nabla_x u(x)|^2)dx\nn\\
    &\leq C \frac{1}{\tilde{\omega}_V(h)^{1/2}h^{2/3}}\int_{B(0,R_0)} e^{2\tilde{\phi}(x)\,\tilde{\omega}_V(h)^{1/2}/ h^{4/3}}|(-h^2\Delta_x + V_{per}- E) u|^2dx\nn\\
    &\quad + C \frac{1}{\tilde{\omega}_V(h)^{1/2}h^{2/3}}\int_{B(0,R_0)} e^{2\tilde{\phi}(x)\,\tilde{\omega}_V(h)^{1/2}/ h^{4/3}}|[-h^2\Delta_x ,\chi] u|^2dx
\end{align}

The support of $\nabla\chi$ is in $A(0,\tilde{R}_0,R_0)$ so the support of $[-h^2\Delta_x ,\chi] u$ is also in $A(0,\tilde{R}_0,R_0)$. Moreover there is $\beta>0$ such that:

\begin{equation}
|[-h^2\Delta_x ,\chi] u|^2 \leq  h^2 \beta(|u|^2 +|h\nabla_x u|^2)
\end{equation} 

Thus we have

\begin{align}
    &\int_{B(0,\tilde{R}_0)} e^{2\tilde{\phi}(x)\,\tilde{\omega}_V(h)^{1/2}/ h^{4/3}} (|u(x)|^2 +|h\nabla_x u(x)|^2)dx\nn\\
    &\leq \tilde{C} \frac{1}{\tilde{\omega}_V(h)^{1/2}h^{2/3}}\int_{B(0,R_0)} e^{2\tilde{\phi}(x)\,\tilde{\omega}_V(h)^{1/2}/ h^{4/3}}|(-h^2\Delta_x + V_{per}- E)u|^2dx\nn\\
    &\quad +\tilde{C} \frac{h^{4/3}}{\tilde{\omega}_V(h)^{1/2}}\int_{A(0,\tilde{R}_0,R_0)} e^{2\tilde{\phi}(x)\,\tilde{\omega}_V(h)^{1/2}/ h^{4/3}} (|u(x)|^2 +|h\nabla_x u(x)|^2)dx
\end{align}
with $\tilde{C}$. Hence the first result.\\
For the second case, thanks to the corollary \ref{cor: cor1} we know that for $u \in C_0^\infty(B(0,R_0)\setminus B(0,R/2))$, we have for $h$ small enough:

\begin{align}
    &\int_{B(0,R_0)} e^{2\tilde{\phi}(x)\,\tilde{\omega}_V(h)^{1/2}/ h^{4/3}} (|u(x)|^2 +|h\nabla_x u(x)|^2)dx\nn\\
    &\leq C \frac{1}{\tilde{\omega}_V(h)^{1/2}h^{2/3}}\int_{B(0,R_0)} e^{2\tilde{\phi}(x)\,\tilde{\omega}_V(h)^{1/2}/ h^{4/3}}|(-h^2\Delta_x + V_{per}- E)u|^2dx
\end{align}

Let's replace $u$ by a smooth function and let $\chi$ be in $ C_0^\infty(B(0,R_0)\setminus B(0,R/2))$. Then we can apply corollary \ref{cor: cor1} to $\chi u$:

\begin{align}
    &\int_{B(0,R_0)} e^{2\tilde{\phi}(x)\,\tilde{\omega}_V(h)^{1/2}/ h^{4/3}} (|\chi u(x)|^2 +|h\nabla_x \chi u(x)|^2)dx\nn\\
    &\leq C \frac{1}{\tilde{\omega}_V(h)^{1/2}h^{2/3}}\int_{B(0,R_0)} e^{2\tilde{\phi}(x)\,\tilde{\omega}_V(h)^{1/2}/ h^{4/3}}|(-h^2\Delta_x + V_{per}- E)\chi u|^2dx
\end{align}

Suppose that $\chi=1$ on $A(0,R_1,R_2)$ with $R/2<R_1<R<R_2<R_0$ and $\chi\geq 0$. Then:

\begin{align}
    &\int_{A(0,R_1,R_2)} e^{2\tilde{\phi}(x)\,\tilde{\omega}_V(h)^{1/2}/ h^{4/3}} ( |u(x)|^2 +|h\nabla_x u(x)|^2)dx\nn\\
    &\leq C \frac{1}{\tilde{\omega}_V(h)^{1/2}h^{2/3}}\int_{B(0,R_0)} \chi e^{2\tilde{\phi}(x)\,\tilde{\omega}_V(h)^{1/2}/ h^{4/3}}|(-h^2\Delta_x + V- E) u|^2dx\nn\\
    &\quad + C \frac{1}{\tilde{\omega}_V(h)^{1/2}h^{2/3}}\int_{B(0,R_0)} e^{2\tilde{\phi}(x)\,\tilde{\omega}_V(h)^{1/2}/ h^{4/3}}|[-h^2\Delta_x ,\chi] u|^2dx
\end{align}

The support of $\nabla\chi$ is in $A(0,R_2,R_0)\cup A(0,r/2,R_1)$ so the support of $[-h^2\Delta_x ,\chi] u$ is also in $A(0,R_2,R_0)\cup A(0,R/2,R_1)$. Moreover there is $\beta>0$

\begin{equation}
|[-h^2\Delta_x ,\chi] u|^2 \leq  h^2 \beta(|u|^2 +|h\nabla_x u|^2)
\end{equation} 

Thus we have

\begin{align}
    &\int_{A(0,R_1,R_2)} e^{2\tilde{\phi}(x)\,\tilde{\omega}_V(h)^{1/2}/ h^{4/3}} (|u(x)|^2 +|h\nabla_x u(x)|^2)dx\nn\\
    &\leq \tilde{C} \frac{1}{\tilde{\omega}_V(h)^{1/2}h^{2/3}}\int_{B(0,R_0)} \chi e^{2\tilde{\phi}(x)\,\tilde{\omega}_V(h)^{1/2}/ h^{4/3}}|(-h^2\Delta_x + V- E)u|^2dx\nn\\
    &\quad +\tilde{C} \frac{h^{4/3}}{\tilde{\omega}_V(h)^{1/2}}\int_{A(0,R_2,R_0)} e^{2\tilde{\phi}(x)\,\tilde{\omega}_V(h)^{1/2}/ h^{4/3}} (|u(x)|^2 +|h\nabla_x u(x)|^2)dx\nn\\
    &\quad +\tilde{C} \frac{h^{4/3}}{\tilde{\omega}_V(h)^{1/2}}\int_{A(0,r/2,R_1)} e^{2\tilde{\phi}(x)\,\tilde{\omega}_V(h)^{1/2}/ h^{4/3}} (|u(x)|^2 +|h\nabla_x u(x)|^2)dx
\end{align}
with $\tilde{C}>0$. On $A(0,R_1,R_0)$, $\tilde{\phi} \leq \tilde{\phi}(R_1)=:K$. Eventually, we obtain:

\begin{align}
    &\int_{A(0,R_1,R_2)} e^{2\tilde{\phi}(x)\,\tilde{\omega}_V(h)^{1/2}/ h^{4/3}} (|u(x)|^2 +|h\nabla_x u(x)|^2)dx\nn\\
    &\leq \tilde{C} \frac{1}{\tilde{\omega}_V(h)^{1/2}h^{2/3}}\int_{B(0,R_0)} e^{2K\,\tilde{\omega}_V(h)^{1/2}/ h^{4/3}}|(-h^2\Delta_x + V- E)u|^2dx\nn\\
    &\quad + \tilde{C} \frac{h^{4/3}}{\tilde{\omega}_V(h)^{1/2}}\int_{B_r} e^{2K\,\tilde{\omega}_V(h)^{1/2}/ h^{4/3}} (|u(x)|^2 +|h\nabla_x u(x)|^2)dx\nn\\
    &\quad +\tilde{C} \frac{h^{4/3}}{\tilde{\omega}_V(h)^{1/2}}\int_{A(0,R_2,R_0)} e^{2\tilde{\phi}(x)\,\tilde{\omega}_V(h)^{1/2}/ h^{4/3}} (|u(x)|^2 +|h\nabla_x u(x)|^2)dx
\end{align}
Hence the second result.

\end{proof}
\subsection{Conclusion}
So from the beginning, we work on a closed Riemannian surface $M$ and we look at the equation:

\begin{equation}
(-h^2\Delta +V-E)u=f \quad \mbox{on M}
\end{equation}

where  $V\in L^\infty(M,\R)$, $u\in H^2(M,\R)$ and $f\in L^2(M,\R)$. Let $U\subset M$ an open subset.
In the subsection \ref{subsect: Uniformization}, we have seen that we can reduce or study to only two cases on ball $B(0,R_0)\subset \R^2$: with centre $0$ and radius $R_0>0$:
\begin{itemize}
	\item $U$ is replace by a ball $B(0,R)$ with radius $0<R<R_0$ .
	\item $U$ is replace by a ring $A(0,\tilde{R_0},R_0)$ with $0<\tilde{R}_0<R_0$.
\end{itemize}
$u,\,V,\,f$ are replaced by "periodic" function  $u_{per},\, V_{per},\, f_{per}$ in the sense of transformation see in the subsection \ref{subsect: Uniformization}:
$$(-h^2\Delta +V_{per}-E)u_{per}=f_{per} \quad \mbox{on } B(0,R_0).$$

For the first case, we use the second statement of the corollary \ref{cor: cor2}:

\begin{align}
   &\int_{A(0,R_1,R_2)} e^{2\tilde{\phi}(x)\,\tilde{\omega}_V(h)^{1/2}/ h^{4/3}} (|u_{per}(x)|^2 +|h\nabla_x u_{per}(x)|^2)dx\nn\\
   &\leq C \frac{1}{\tilde{\omega}_V(h)^{1/2}h^{2/3}}\int_{B(0,R_0)} e^{2K\,\tilde{\omega}_V(h)^{1/2}/ h^{4/3}}|f_{per}(x)|^2dx\nn\\
   &\quad + C \frac{h^{4/3}}{\tilde{\omega}_V(h)^{1/2}}\int_{B(0,R)} e^{2K\,\tilde{\omega}_V(h)^{1/2}/ h^{4/3}} (|u_{per}(x)|^2 +|h\nabla_x u_{per}(x)|^2)dx\nn\\
   &\quad +C \frac{h^{4/3}}{\tilde{\omega}_V(h)^{1/2}}\int_{A(0,R_2,R_0)} e^{2\tilde{\phi}(x)\,\tilde{\omega}_V(h)^{1/2}/ h^{4/3}} (|u_{per}(x)|^2 +|h\nabla_x u_{per}(x)|^2)dx
\end{align}
 for $h$ small enough, with $\tilde{\phi}(x)$ a smooth positive and radially decreasing function, $C>0$, $K>\tilde{\phi}$ on $A(0,R_1,R_0)$ a ring centre on $0$ with radius $R/2<R_1<R<R_2<R_0$ and
with

\begin{equation}
\tilde{\omega}_V(h)= \frac{1}{h^{4/3}}\sup_{x_0\in B(0,R_0)}\sup_{x\in B(x_0,h^{2/3}\kappa)}|V(x)-V(x_0)|^{1/2}
\end{equation}

Thanks to the "periodization" and the fact that $\tilde{\phi}(x)$ is decreasing:

\begin{align}
	\int_{A(0,R_2,R_0)} e^{2\tilde{\phi}(x)\,\tilde{\omega}_V(h)^{1/2}/ h^{4/3}} (|u_{per}(x)|^2 +|h\nabla_x u_{per}(x)|^2)dx\nn\\
	= \mathcal{O}\left(\int_{B(0,R_2)} e^{2\tilde{\phi}(x)\,\tilde{\omega}_V(h)^{1/2}/ h^{4/3}} (|u_{per}(x)|^2 +|h\nabla_x u_{per}(x)|^2)dx\right)
\end{align}
if $1<<R_2<R_0$.

So we have for $h$ small enough:
\begin{align}
    &C \frac{h^{4/3}}{\tilde{\omega}_V(h)^{1/2}}\int_{A(0,R_1,R_2)} e^{2\tilde{\phi}(x)\,\tilde{\omega}_V(h)^{1/2}/ h^{4/3}} (|u_{per}(x)|^2 +|h\nabla_x u_{per}(x)|^2)dx \nn\\
    &\leq C \frac{h^{4/3}}{\tilde{\omega}_V(h)^{1/2}}\int_{B(0,R)} e^{2K\,\tilde{\omega}_V(h)^{1/2}/ h^{4/3}} (|u_{per}(x)|^2 +|h\nabla_x u_{per}(x)|^2)dx\nn\\
    &+\int_{A(0,R_1,R_2)} e^{2\tilde{\phi}(x)\,\tilde{\omega}_V(h)^{1/2}/ h^{4/3}} (|u_{per}(x)|^2 +|h\nabla_x u_{per}(x)|^2)dx
\end{align}

Thus:

\begin{align}
\int_{B(0,R_0)} (|u_{per}(x)|^2 &+|h\nabla_x u_{per}(x)|^2)dx \ \nn\\
&\leq \int_{B(0,R)} (|u_{per}(x)|^2 +|h\nabla_x u_{per}(x)|^2)dx\nn\\
&\quad + \int_{A(0,R_1,R_2)} (|u_{per}(x)|^2 +|h\nabla_x u_{per}(x)|^2)dx\nn\\
&\quad +\int_{A(0,R_2,R_0)} (|u_{per}(x)|^2 +|h\nabla_x u_{per}(x)|^2)dx\nn\\
&\leq \tilde{C} e^{\tilde{C}\,\tilde{\omega}_V(h)^{1/2}/ h^{4/3}} \int_{B(0,R)}  (|u_{per}(x)|^2 +|h\nabla_x u_{per}(x)|^2)dx\nn\\
&\quad + \tilde{C} e^{\tilde{C}\,\tilde{\omega}_V(h)^{1/2}/ h^{4/3}}\int_{B(0,R_0)} |f_{per}(x)|^2dx
\end{align}
with $\tilde{C}>0$
By periodicity, we can replace $\tilde{\omega}_V(h)$ by $\beta(h)= \frac{1}{h^{4/3}}\sup_{x_0\in M}\sup_{x\in B(x_0,h^{2/3}\kappa)}|V(x)-V(x_0)|^{1/2}$. By the subsection \ref{subsect: Uniformization}, we can reduce the integral on $B(0,R_0)$ to the integral on $M$ and replace $u_{per}$ and $f_{per}$:

\begin{align}
\int_{M}|u|^2+|h\nabla u|^2&\leq \tilde{C} e^{\tilde{C}\,\beta(h)^{1/2}/ h^{4/3}}\left(\int_{M} |f|^2+\int_{U} (|u|^2 +|h\nabla u|^2)\right)
\end{align}
This conclude the first case.

For the second one, we use here the first point of corollary \ref{cor: cor2}:

\begin{align}
       &\int_{B(0,\tilde{R}_0)} e^{2\tilde{\phi}(x)\,\tilde{\omega}_V(h)^{1/2}/ h^{4/3}} (|u_{per}(x)|^2 +|h\nabla_x u_{per}(x)|^2)dx\nn\\
    &\leq C \frac{1}{\tilde{\omega}_V(h)^{1/2}h^{2/3}}\int_{B(0,R_0)} e^{2\tilde{\phi}(x)\,\tilde{\omega}_V(h)^{1/2}/ h^{4/3}}|f_{per}(x)|^2dx\nn\\
    &\quad +C \frac{h^{4/3}}{\tilde{\omega}_V(h)^{1/2}}\int_{A(0,\tilde{R}_0,R_0)} e^{2\tilde{\phi}(x)\,\tilde{\omega}_V(h)^{1/2}/ h^{4/3}} (|u_{per}(x)|^2 +|h\nabla_x u_{per}(x)|^2)dx
\end{align}

for $h$ small enough, with $\tilde{\phi}(x)$ a smooth positive function, $C>0$, $A(0,\tilde{R}_0,R_0)$ is a ring centre on $0$ with radius $0<\tilde{R}_0<R_0$ and with

\begin{equation}
\tilde{\omega}_V(h)= \frac{1}{h^{4/3}}\sup_{x_0\in B(0,R_0)}\sup_{x\in B(x_0,h^{2/3}\kappa)}|V(x)-V(x_0)|^{1/2}.
\end{equation}

So:

\begin{align}
	\int_{B(0,R_0)}|u_{per}(x)|^2&+|h\nabla_x u_{per}(x)|^2\nn\\
	&=\int_{B(0,\tilde{R}_0)}|u_{per}(x)|^2+|h\nabla_x u_{per}(x)|^2+\int_{A(0,\tilde{R}_0,R_0)}|u_{per}(x)|^2+|h\nabla_x u_{per}(x)|^2\nn\\
	&\leq (C \frac{h^{4/3}}{\tilde{\omega}_V(h)^{1/2}}+1)\int_{A(0,\tilde{R}_0,R_0)} e^{2\tilde{\phi}(x)\,\tilde{\omega}_V(h)^{1/2}/ h^{4/3}} (|u_{per}(x)|^2 +|h\nabla_x u_{per}(x)|^2)dx\nn\\
	&\quad +C \frac{1}{\tilde{\omega}_V(h)^{1/2}h^{2/3}}\int_{B(0,R_0)} e^{2\tilde{\phi}(x)\,\tilde{\omega}_V(h)^{1/2}/ h^{4/3}}|f_{per}(x)|^2dx\nn\\
	& \leq \tilde{C} e^{\tilde{C}\,\tilde{\omega}_V(h)^{1/2}/ h^{4/3}}\int_{A(0,\tilde{R_0},R_0)} (|u_{per}(x)|^2 +|h\nabla_x u_{per}(x)|^2)dx\nn\\
	& \quad+\tilde{C} e^{\tilde{C}\,\tilde{\omega}_V(h)^{1/2}/ h^{4/3}}\int_{B(0,R_0)}|f_{per}(x)|^2dx
\end{align}
with $\tilde{C}>0$.
By periodicity, we can replace $\tilde{\omega}_V(h)$ by $\beta(h)= \frac{1}{h^{4/3}}\sup_{x_0\in M}\sup_{x\in B(x_0,h^{2/3}\kappa)}|V(x)-V(x_0)|^{1/2}$.
Then by the subsection \ref{subsect: Uniformization}, we can reduce the integral on $B(0,R_0)$ to the integral on $M$ and replace $u_{per}$ and $f_{per}$:

\begin{align}
\int_{M}|u|^2+|h\nabla u|^2&\leq \tilde{C} e^{\tilde{C}\,\beta(h)^{1/2}/ h^{4/3}}\left(\int_{M}|f|+\int_{U} (|u|^2 +|h\nabla u|^2)\right)
\end{align}
Which conclude the second case.

\bibliographystyle{elsarticle-harv} 
\typeout{}
\bibliography{refs}

@article{kloppSemiclassicalResolventEstimates2019a,
  title = {Semiclassical Resolvent Estimates for Bounded Potentials},
  author = {Klopp, Fr{\'e}d{\'e}ric and Vogel, Martin},
  year = {2019},
  month = jan,
  journal = {Pure and Applied Analysis},
  volume = {1},
  number = {1},
  pages = {1--25},
  issn = {2578-5885, 2578-5893},
  doi = {10.2140/paa.2019.1.1},
  urldate = {2025-01-25},
  langid = {english}
}

@article{vodevSemiclassicalResolventEstimates2020,
  title = {Semiclassical Resolvent Estimates for {{H{\"o}lderpotentials}}},
  author = {Vodev, Georgi},
  year = {2020},
  month = dec,
  journal = {Pure and Applied Analysis},
  volume = {2},
  number = {4},
  pages = {841--860},
  issn = {2578-5885, 2578-5893},
  doi = {10.2140/paa.2020.2.841},
  urldate = {2024-07-15},
  langid = {english}
}

@book{alinhac2007pseudo,
  title = {Pseudo-differential Operators and the Nash–Moser Theorem},
  ISBN = {9781470421120},
  ISSN = {1065-7339},
  url = {http://dx.doi.org/10.1090/gsm/082},
  DOI = {10.1090/gsm/082},
  journal = {Graduate Studies in Mathematics},
  publisher = {American Mathematical Society},
  author = {Alinhac,  Serge and Gérard,  Patrick},
  editor = {Wilson,  Stephen},
  year = {2007},
  month = apr 
}

@article{de2016uniformization,
  title = {Uniformization of Riemann Surfaces: Revisiting a hundred-year-old theorem},
  ISBN = {9783037196458},
  ISSN = {2523-5222},
  url = {http://dx.doi.org/10.4171/145},
  DOI = {10.4171/145},
  journal = {Heritage of European Mathematics},
  publisher = {EMS Press},
  author = {de Saint-Gervais,  Henri Paul},
  year = {2016},
  month = jan 
}

@book{dyatlovMathematicalTheoryScattering2019,
  title = {Mathematical {{Theory}} of {{Scattering Resonances}}},
  author = {Dyatlov, Semyon and Zworski, Maciej},
  year = {2019},
  month = sep,
  series = {Graduate {{Studies}} in {{Mathematics}}},
  volume = {200},
  publisher = {American Mathematical                     Society},
  address = {Providence, Rhode                     Island},
  doi = {10.1090/gsm/200},
  urldate = {2025-10-02},
  isbn = {978-1-4704-4366-5 978-1-4704-5313-8},
  langid = {english}
}

\end{document}